\documentclass{article}
\usepackage{array,amscd,amsfonts,amsmath,amssymb,amsthm,graphics,mathrsfs,multirow,verbatim,enumitem,hyperref,algorithm}
\usepackage[noend]{algpseudocode}
\usepackage[utf8]{inputenc}
\usepackage[english]{babel}

\begingroup
    \makeatletter
    \@for\theoremstyle:=plain,definition,remark\do{%
        \expandafter\g@addto@macro\csname th@\theoremstyle\endcsname{%
            \addtolength\thm@preskip\parskip
            }%
        }
\endgroup
\usepackage[parfill]{parskip}
\usepackage{geometry}
\theoremstyle{plain}
\newtheorem{theorem}{Theorem}

\newtheorem{proposition}[theorem]{Proposition}
\theoremstyle{definition}
\newtheorem{definition}{Definition}

\newtheorem{strategy}{Strategy}
\theoremstyle{remark}
\newtheorem{remark}{Remark}

\begin{document}
\sloppy

\title{\Large{\textbf{A GENERALIZED COVER'S PROBLEM}}}
\author{\normalsize{BENJAMIN E. DIAMOND}}
\date{}
\maketitle

\begin{abstract}
Generalizing a problem posed by Cover \cite{Cover:1987aa}, we propose an adversarial game in which a permutation is incrementally constructed in a setting of partial information. As in the secretary problem, this permutation is exposed in stages via the successive components of its Lehmer code. Extending Cover's result, which constitutes the case $n = 2$, we establish that a random permutation of $n$ adversarially constructed real numbers can be reconstructed with better-than-random probability, provided that certain among the numbers it permutes are made visible during the process.
\end{abstract}

\section{Introduction}
\begin{quote}
Player 1 writes down any two distinct numbers on separate clips of paper. Player 2 randomly chooses one of these slips of paper and looks at the number. Player 2 must decide whether the number in his hand is the larger of the two numbers. He can be right with probability one-half. It seems absurd that he can do better.

We argue that Player 2 has a strategy by which he can correctly state whether or not the other number is larger or smaller than the number in his hand with probability \textit{strictly greater than one-half}.
\end{quote}
Thomas Cover's well-known paradox \cite{Cover:1987aa}, reproduced in part above, asserts that one can identify the larger of two unequal numbers chosen by an adversary, having seen only one of them, with probability surpassing that of a randomized strategy.

This setting has been generalized in a number of directions. Most prominent among these is the \textit{secretary problem}, or, more properly, the game of \textit{googol}---distinguished by Player 2's seeing the actual values of the numbers as they come, as opposed their relative ranks alone---identified as such apparently by Ferguson \cite{Ferguson:1989aa}, and solved ultimately by Samuels \cite{Samuels:1981aa}, Ferguson \cite{Ferguson:1989aa}, Silverman and N\'{a}das \cite{Silverman:1990aa}, and Gnedin \cite{Gnedin:1994aa}, among others. The problem's classical solution, based only on relative ranks, performs well; as long as $n > 2$, however, no additional---or rather, actionable---information is conveyed by the numbers' actual values (see especially \cite[4. Final remarks]{Gnedin:1994aa}).

Further extensions are surveyed in Ferguson \cite{Ferguson:1989aa}. Modern generalizations include Kesselheim, Radke, T\"{o}nnis, and V\"{o}cking \cite{Kesselheim:2013aa}, which studies weighted matching to multiple targets, and Gnedin \cite{Gnedin:2016aa}, which studies guessing which among two \textit{piles} of numbers \textit{contains} the largest number.

The extension of Cover's problem to the game of googol entails viewing Player 2's choice as that mandated by a stopping procedure, which seeks, at each stage $k$, for $k \in \{1, \ldots , n - 1\}$, to infer properties of an unknown permutation $\sigma \in \mathbf{S}_n$ given only the first $k$ elements of its Lehmer code. Precisely, in googol, Player 2 seeks to stop exactly when $k$ is such that $\sigma(n) = k$.

We propose a different generalization, in which, at each such stage $k$, Player 2 attempts to determine the \textit{next} element of $\sigma$'s Lehmer code. (This version also recovers Cover's problem in the case $n = 2$.) About this $k + 1$\textsuperscript{st} element, of course, no information is conveyed by the first $k$. Contrarily to the case of googol, however, access to the numbers, in this problem, continues to bestow an advantage, even for $n$ greater than $2$. Indeed, generalizing Cover's result, we demonstrate that, for arbitrary $n \geq 2$, $n$ adversarially generated real numbers permuted by a uniformly random permutation $\sigma$ can be used to reconstruct $\sigma$ (through its Lehmer code, in $n - 1$ steps) with probability greater than $\frac{1}{n!}$, provided that, at each stage $k$, the first $k$ of these permuted numbers are made visible.

This direction carries forward the spirit of Cover's original observation, and aims to answer his question: ``Does this result generalize?'' In an appendix, we discuss further extensions to the case $n = 2$.

\subsection{The classical case $n = 2$}
We summarize the classical treatment of Cover. Cover suggests the following strategy \cite[Solution]{Cover:1987aa}:

\begin{strategy}[Cover] \label{cover}
Player 2 fixes a real-valued random variable $T$ which is absolutely continuous with everywhere-positive density. Selecting a random variate $t$ from $T$, and calling the number he chooses $x$, Player 2 states that the other number is smaller than $x$ just in case $t \leq x$. Otherwise, he states that it is larger.
\end{strategy}

We establish the correctness of Strategy \ref{cover} \textit{a fortiori} by realizing it as a special case of a slightly more general strategy:

\begin{strategy} \label{increasing}
Player 2 picks a strictly increasing function $f \colon \mathbb{R} \rightarrow [0, 1]$. Calling the number he chooses $x$, Player 2 states that the other number is smaller than $x$ with probability $f(x)$. Otherwise, he states that it is larger.
\end{strategy}

Any instance of Strategy \ref{cover}, with random variable $T$ say, can be realized as an instance of Strategy \ref{increasing} by setting $T$'s cumulative distribution function as $f$.

\begin{proposition} \label{case2}
Strategy \ref{increasing} is dominant for Player 2 in the above game.
\end{proposition}
\begin{proof}
We name Player 1's two numbers, which are distinct, $s$ and $l$, for \textit{smaller} and \textit{larger}. Breaking down the probability of winning along the two outcomes of Player 2's random choice, we have that:
\begin{align*}
P(\text{win}) &= P(\text{chooses } l) \cdot P(\text{states ``smaller''}) + P(\text{chooses } s) \cdot P(\text{states ``larger''}) \\
&= \frac{1}{2} \cdot f(l) + \frac{1}{2} \left( 1 - f(s) \right) \\
&= \frac{1}{2} + \frac{f(l) - f(s)}{2} \\
&> \frac{1}{2},
\end{align*}
where in the final step we use the increasingness of $f$.
\end{proof}

\subsection{Notations and terminology}

We prepare definitions which will be key in what follows. We fix an $n \geq 2$ once and for all.

We begin with the following correspondence, often called the \textit{Lehmer code}, and originating apparently with Hall (see Sedgewick \cite[pp. 154-155]{Sedgewick:1977aa}):

\begin{definition}[Lehmer code]
To each element $\sigma \in \mathbf{S}_n$, viewed as a bijection on $\{1, \ldots , n\}$, we associate a tuple $(c_1, \ldots , c_n)$, which we will call $\sigma$'s \textit{Lehmer code}, by setting $c_1 = 0$ and declaring, for each $k \in \{1, \ldots , n - 1\}$, that $c_{k + 1} \in \{0, \ldots , k \}$ be the cardinality
\begin{equation*}\left| \left\{ i \in \{1, \ldots , k \} \mid \sigma^{-1}(i) < \sigma^{-1}(k + 1) \right\} \right|.\end{equation*}
Conversely, any such tuple $(c_1, \ldots , c_n)$ corresponds to the permutation given by $\sigma \colon i \mapsto P[i]$, where the array $P[i], i \in \{1, \ldots , n \}$ is defined through the following procedure (see Sedgewick \cite[p. 155]{Sedgewick:1977aa}):
\begin{algorithm}[H]
\begin{algorithmic}[1]
\Procedure{Lehmer}{$c$, $P$}:
\For{$i = 1 \text{ to } n$}:
\For{$j = 1 \text{ to } c_i$}:
\State $P[n - i + j] \gets P[n - i + j + 1]$
\EndFor
\State $P[n - i + c_i + 1] \gets i$
\EndFor
\EndProcedure
\end{algorithmic}
\end{algorithm}

In what follows, we often freely identify elements $\sigma \in \mathbf{S}_n$ with Lehmer codes $(c_1, \ldots , c_n)$.
\end{definition}

Intuitively, each successive element $c_{k + 1}$ encodes the relative position with respect to the numbers $\{1, \ldots , k\}$ that $k + 1$ occupies, in the list $\sigma(1), \ldots , \sigma(n)$. An example is visible in Sedgewick \cite[p. 155]{Sedgewick:1977aa}.

\begin{definition}
By the \textit{standard $k$-simplex} $\Delta^k$ we will mean the set:
\begin{equation*}\Delta^k = \left\{ (p_0, \ldots , p_k) \in \mathbb{R}^{k + 1} \mid \sum_{i = 0}^k p_i = 1, p_i \geq 0 \text{ for all } i \right\}.\end{equation*}
We identify categorical distributions on $\{0, \ldots , k\}$ with elements of this set.
\end{definition}

\section{Results}

Continuing an apparent tradition, we describe the key game in verbal form.

\begin{quote}
Player 1 writes down any $n$ distinct numbers on separate clips of paper. Player 2 randomly chooses one of these slips of paper and looks at its number. As long as there remain further slips on the table, now, Player 2 repeatedly enacts a procedure whereby, randomly selecting a further slip, he must guess its hidden number's ordinal position with respect to the collection of slips which are already visible, before, finally, revealing this slip's number, continuing to the next only if his guess had been correct.

We argue that Player 2 has a strategy by which he can correctly order the entire collection of $n$ slips with probability \textit{strictly greater than 1 divided by $n$ factorial}.
\end{quote}

We enrich this game with definitions. Writing without loss of generality the adversary's (real) numbers as $x_1 > \cdots > x_n$, we denote by $\sigma \in \mathbf{S}_n$ that unique permutation for which these slips' numbers are, in the order in which they are revealed, $x_{\sigma^{-1}(1)}, \ldots , x_{\sigma^{-1}(n)}$. We index the stages of the game by $k \in \{1, \ldots , n - 1\}$. At each stage $k$, then, the numbers $x_{\sigma^{-1}(1)}, \ldots , x_{\sigma^{-1}(k)}$ are visible, as are, hence, the elements $(c_1, \ldots , c_k)$ of $\sigma$'s Lehmer code. Player 2's task in this stage $k$ is exactly to guess, on the basis of these visible numbers, the element $c_{k + 1}$ of $\sigma$'s Lehmer code. Player 2 wins if he completely reconstructs $\sigma$.

We point out that strategies in this game are effectively family of maps $\left\{ f_k \colon \mathbb{R}^k \rightarrow \Delta^k \right\}_{k \in \{1, \ldots , n - 1\}}$. We propose the following particular strategy:

\begin{strategy} \label{strategy}
At each stage $k \in \{1, \ldots , n - 1\}$, Player 2 guesses $c_{k + 1}$ on the basis of a sample from that categorical distribution on $\{0, \ldots , k\}$ given as the image of $x_{\sigma^{-1}(1)}, \ldots , x_{\sigma^{-1}(k)}$ under the map:
\begin{equation*}f_k \colon \left( \eta_1, \ldots , \eta_k \right) \mapsto \text{softmax} \left( 0, \eta_{\sigma_k(1)} \ldots , \eta_{\sigma_k(k)} \right) = \frac{(1, e^{\eta_{\sigma_k(1)}}, \ldots , e^{\eta_{\sigma_k(k)}})}{1 + e^{\eta_{\sigma_k(1)}} + \cdots + e^{\eta_{\sigma_k(k)}}},\end{equation*}
where, for each $(\eta_1, \ldots , \eta_k) \in \mathbb{R}^k$, $\sigma_k \in \mathbf{S}_k$ is chosen to be that unique permutation for which $\eta_{\sigma_k(1)} > \cdots > \eta_{\sigma_k(k)}$. In other words, at each stage, Player 2 sorts his visible numbers in descending order, before prepending a leading zero and applying the softmax.
\end{strategy}

\begin{remark}
The permutation $\sigma_k \in \mathbf{S}_k$ as defined above coincides with that given by the truncated Lehmer code $(c_1, \ldots , c_k)$ available as of stage $k$.
\end{remark}

\begin{remark} \label{genstrat}
Considering in particular the case $k = 1$, we point out that $f_1$ coincides with the logistic function $x_1 \mapsto \frac{e^{x_1}}{1 + e^{x_1}}$, provided that the $1$-simplex is identified with the unit interval $[0, 1]$. As this function is increasing, we see that Strategy \ref{strategy} generalizes this particular instance of Strategy \ref{increasing}.
\end{remark}

\begin{remark} \label{exponential}
This strategy appears related to the theory of exponential families. Indeed, interpreting the point $\left( x_{\sigma^{-1}(1)}, \ldots , x_{\sigma^{-1}(k)} \right) = (\eta_1, \ldots , \eta_k) = \left( \text{log}\left(\frac{p_1}{p_0}\right), \ldots , \text{log}\left(\frac{p_k}{p_0}\right) \right) \in \mathbb{R}^k$ as the natural parameter of an underlying categorical distribution $(p_0, \ldots , p_k) \in \Delta^k$, we realize Strategy \ref{strategy} as a map from the $k$-simplex to itself (we refer for example to \cite[1.5, Ex. 5.3 and 3.6, (6.2)]{Lehmann:1998aa}). This map is exactly that which sorts in descending order the trailing (that is, nonzero-indexed) coordinates.
\end{remark}

We come to the main theorem:

\begin{theorem}
Strategy \ref{strategy} is dominant for Player 2 in the generalized game above.
\end{theorem}
\begin{proof}
We show that, for any numbers $x_1 > \cdots > x_n$, the uniform weighted average
\begin{equation}\sum_{\sigma \in \mathbf{S}_n} \frac{1}{n!} \cdot \prod_{k = 1}^{n - 1} f_k \left( x_{\sigma^{-1}(1)}, \ldots , x_{\sigma^{-1}(k)} \right)(c_{k + 1}) > \frac{1}{n!},\label{sop} \end{equation}
where, for each $\sigma$ and at each stage $k$, $c_{k + 1}$ is the (as yet unknown) $k + 1$\textsuperscript{st} component of $\sigma$'s Lehmer code. We multiply both sides of this equation by $\frac{1}{n!}$ once and for all. Noting also that the case $n = 2$ is proven by Proposition \ref{case2} above, we assume that $n \geq 3$ in what follows.

We observe first that as, at each stage $k$, $f_k$ immediately sorts its arguments $x_{\sigma^{-1}(1)}, \ldots , x_{\sigma^{-1}(k)}$, we may as well view $f_k$ as a function on the unordered collection of these arguments.

We now note that, for each fixed $d_n \in \{0, \ldots , n - 1\}$, those permutations $\sigma = (c_1, \ldots , c_n) \in \mathbf{S}_n$ for which $c_n = d_n$ yield identical (unordered) sets $\{x_{\sigma^{-1}(1)}, \ldots , x_{\sigma^{-1}(n - 1)}\}$ (namely, they equal $\{x_1, \ldots , x_n\} \backslash \{x_{d_n + 1}\}$). As a first consequence, we rewrite the left-hand side of the inequality (\ref{sop}):
\begin{equation}\sum_{d_n \in \{0, \ldots , n - 1\}} \left( \sum_{\substack{\sigma = (c_1, \ldots , c_n) \in \mathbf{S}_n \\ c_n = d_n}} \prod_{k = 1}^{n - 2} f_k \left( x_{\sigma^{-1}(1)}, \ldots , x_{\sigma^{-1}(k)} \right)(c_{k + 1}) \right) \cdot f_{n - 1} \left( \{x_1, \ldots , x_n\} \backslash \{x_{d_n + 1}\} \right)(d_n). \label{firstbreakdown}\end{equation}

In fact, for any fixed \textit{trailing Lehmer code} $(d_{m + 1}, \ldots , d_n)$, $m \in \{1, \ldots , n - 1\}$, those permutations $\sigma = (c_1, \ldots , c_n) \in \mathbf{S}_n$ for which $c_{m + 1} = d_{m + 1}, \ldots , c_n = d_n$ yield identical chains of sets $\{x_{\sigma^{-1}(1)}, \ldots , x_{\sigma^{-1}(m)}\} \subset \cdots \subset \{x_{\sigma^{-1}(1)}, \ldots , x_{\sigma^{-1}(n)}\}$ (for which, however, when $m < n - 1$ no convenient closed-form expression appears available).  

Using this observation, we exhibit a general recursive substructure of the form (\ref{firstbreakdown}). (We assume $n > 3$ in what follows, for notational convenience. The case $n = 3$ is effectively exhausted by the expression (\ref{firstbreakdown}) above.) For any trailing code $(d_{m + 1}, \ldots , d_n)$, $m \in \{3, \ldots , n - 1\}$, then, we have that:
\begin{align} \label{inductive}
\begin{split}
& \sum_{\substack{\sigma = (c_1, \ldots , c_n) \in \mathbf{S}_n \\ c_{m + 1} = d_{m + 1}, \ldots c_n = d_n}} \prod_{k = 1}^{m - 1} f_k \left( x_{\sigma^{-1}(1)}, \ldots , x_{\sigma^{-1}(k)} \right)(c_{k + 1}) =  \\
& \sum_{d_m \in \{0, \ldots , m - 1\}} \left( \sum_{\substack{\sigma = (c_1, \ldots , c_n) \in \mathbf{S}_n \\ c_m = d_m, \ldots c_n = d_n}} \prod_{k = 1}^{m - 2} f_k \left( x_{\sigma^{-1}(1)}, \ldots , x_{\sigma^{-1}(k)} \right)(c_{k + 1}) \right) \cdot f_{m - 1}\left( \left\{ x_{\sigma^{-1}(1)}, \ldots , x_{\sigma^{-1}(m - 1)} \right\} \right)(d_m),
\end{split}
\end{align}
where, for each $d_m \in \{0, \ldots , m - 1\}$, the set $\left\{ x_{\sigma^{-1}(1)}, \ldots , x_{\sigma^{-1}(m - 1)} \right\}$ does not depend on the choice of $\sigma = (c_1, \ldots , c_n) \in \mathbf{S}_n$ for which $c_m = d_m, \ldots c_n = d_n$.

By induction, therefore, (where the base case $m = 3$ is exactly the classical Proposition \ref{case2}, in light of Remark \ref{genstrat}) we may assume that, for each trailing code $(d_{m + 1}, \ldots , d_n)$, $m \in \{3, \ldots , n - 1\}$, each instance $d_m$ of the inner expression above features an inequality:
\begin{equation*}\sum_{\substack{\sigma = (c_1, \ldots , c_n) \in \mathbf{S}_n \\ c_m = d_m, \ldots c_n = d_n}} \prod_{k = 1}^{m - 2} f_k \left( x_{\sigma^{-1}(1)}, \ldots , x_{\sigma^{-1}(k)} \right)(c_{k + 1}) > 1,\end{equation*}
leaving unproven, only, the inductive step
\begin{equation*}\sum_{d_m \in \{0, \ldots m - 1\}} f_{m - 1}\left( \left\{ x_{\sigma^{-1}(1)}, \ldots , x_{\sigma^{-1}(m - 1)} \right\} \right)(d_m) > 1,\end{equation*}
where, again, the components $(d_{m + 1}, \ldots , d_n)$ are fixed.

Each set $\left\{ x_{\sigma^{-1}(1)} , \ldots , x_{\sigma^{-1}(m - 1)} \right\}$ in the above inequality differs from $\left\{ x_{\sigma^{-1}(1)}, \ldots , x_{\sigma^{-1}(m)} \right\}$ only in its lacking some particular element of the latter set (determined by the element $d_m$). This situation thus mimics that of expression (\ref{firstbreakdown}), and for notational convenience (that is, up to a reindexing), we adopt its setting in what follows. We are thus reduced to the inequality:
\begin{equation}\sum_{i = 1}^n f_{n - 1} \left( \{x_1, \ldots , x_n\} \backslash \{x_i\} \right)(i - 1) = \sum_{i = 1}^n \text{softmax}\left( 0, x_1, \ldots , \widehat{x_i}, \ldots , x_n \right)(i - 1) > 1,\label{inequality}\end{equation}
where, in the right-hand sum, a hat indicates that its corresponding argument is removed.

Pulling back the right-hand sum along the (surjective) recoordinatization
\begin{equation*}\mathring{\Delta}^n \rightarrow \mathbb{R}^n, \quad \left( p_0, \ldots , p_n \right) \mapsto \left( \text{log} \left( \frac{p_1}{p_0} \right), \ldots , \text{log} \left( \frac{p_n}{p_0} \right) \right) = \left( x_1, \ldots , x_n \right)\end{equation*}
from the interior of the $n$-simplex to $\mathbb{R}^n$ (cf. Remark \ref{exponential}), the right-hand expression simplifies, yielding the inequality (note that $p_1 > \cdots > p_n$):
\begin{equation*}\sum_{i = 1}^n \frac{p_{i - 1}}{p_0 + p_1 + \cdots + \widehat{p_i} + \cdots + p_n} > 1.\end{equation*}
On the other hand, we have that:
\begin{align*}
\sum_{i = 1}^n \frac{p_{i - 1}}{p_0 + p_1 + \cdots + \widehat{p_i} + \cdots + p_n} &> \sum_{i = 1}^n \frac{p_{i - 1}}{p_0 + \cdots + p_{n - 1} + \widehat{p_n}} & \text{(using } p_i > p_n, i \in \{1, \ldots , n - 1 \} \text{)} \\
&= 1. & \text{(common denominator)}
\end{align*}
This calculation completes the proof. 
\end{proof}

\begin{remark}
We note a correspondence between permutations $\sigma \in \mathbf{S}_n$ and upward paths through the Hasse diagram, or rather the cube graph $Q_n$ (see for example Biggs \cite{Biggs:1974aa}), on (the power set of) $\{x_1, \ldots x_n\}$, mediated by the Lehmer code. To each permutation $\sigma \in \mathbf{S}_n$ we attach that upward path through $Q_n$ given by $\varnothing \rightarrow \{x_{\sigma^{-1}(1)}\} \rightarrow \cdots \rightarrow \{x_{\sigma^{-1}(1)}, \ldots , x_{\sigma^{-1}(n - 1)}\} \rightarrow \{x_{\sigma^{-1}(1)}, \ldots , x_{\sigma^{-1}(n)}\}$. The specification of a trailing code $(d_{m + 1}, \ldots , d_n)$, $m \in \{1, ..., n - 1\}$, then, identifies exactly those $\sigma \in \mathbf{S}_n$ sharing a certain \textit{trailing path} $\{x_{\sigma^{-1}(1)}, \ldots , x_{\sigma^{-1}(m)}\} \rightarrow \cdots \rightarrow \{x_{\sigma^{-1}(1)}, \ldots , x_{\sigma^{-1}(n)}\}$.

The inductive structure (\ref{inductive}) thus relies on the fact that the subgraph beneath any fixed such trailing path is itself a cube graph, namely the cube graph $Q_m$ on $\{x_{\sigma^{-1}(1)}, \ldots , x_{\sigma^{-1}(m)}\}$, and that an inequality of the form (\ref{inequality}) holds regarding those nodes immediately beneath any fixed node $\{x_{\sigma^{-1}(1)}, \ldots , x_{\sigma^{-1}(m)}\}$.
\end{remark}

\appendix
\section{The case $n = 2$: further directions}

In this appendix, we study a further family of extensions of Cover's classical game, in which we grant Player 1 a more flexible form of choice.

A relationship has been noted between Cover's game and the so-called ``two-envelope paradox'' (see for example Clark and Shackel \cite{Clark:2000aa} for a thorough treatment), in which an intuitively compelling ``switching'' argument must be refuted. We note in particular the paper \cite{Samet:2004aa} of Samet, Samet, and Schmeidler, which places these two games into a common framework, whereby, in each case, an adversary begins by selecting a pair of points in the plane which straddles the main diagonal.

In a problem related to the two-envelope paradox, one might seek to explain the insolubility of the game in which:
\begin{quote}
Player 1 writes down any number on a clip of paper. Player 2 takes this slip and looks at the number. Player 1, then, randomly decides whether respectively to write a number which is higher than or lower than the first number onto a second slip. Player 2 must decide whether the number in his hand is the larger of the two numbers.
\end{quote}
in the face of this game's apparent similarity to the original game of Cover.

In light of these considerations, we propose a general adversarial game. Following \cite{Samet:2004aa}, we denote by $A = \left\{ (x_1, x_2) \in \mathbb{R}^2 \mid x_1 < x_2 \right\}$ and $B = \left\{ (x_1, x_2) \in \mathbb{R}^2 \mid x_1 > x_2 \right\}$, respectively, the regions strictly above and below the main diagonal in $\mathbb{R}^2$. We now have:
\begin{quote}
Player 1 selects a pair $(\mathbf{x}_A, \mathbf{x}_B) \in A \times B$, subject to a fixed \textit{ruleset} $R \subset A \times B$. Player 2 randomly chooses one point, say $\mathbf{x}$, from this pair and looks at its $x_1$-coordinate. Player 2 must decide whether $\mathbf{x}$ resides in $A$ or in $B$.

Player 2 seeks a strategy by which he can correctly state whether $\mathbf{x}$ belongs to $A$ or $B$ with probability \textit{strictly greater than one-half}.
\end{quote}

Cover's original game is the special case of this game in which $R = \left\{((a, b), (b, a) \in A \times B \mid a < b \right\}$ is the set consisting of those pairs of points which are mirror reflections about the diagonal. The ``paradoxically unsolvable'' game above corresponds to that $R = \left\{ ((x, x + \epsilon_A), (x, x - \epsilon_B)) \in A \times B \mid \epsilon_A, \epsilon_B > 0 \right\}$ consisting of those pairs $(\mathbf{x}_A, \mathbf{x}_B)$ which occupy a shared vertical line. We have a general question: for which rulesets $R$ does Player 2 have a dominant strategy?

\begin{definition}
To any ruleset $R \subset A \times B$ as defined above, we associate a set $X_R$ equipped with a binary relation $P_R$, defined as:
\begin{equation*}X_R = \bigcup_{(\mathbf{x}_A, \mathbf{x}_B) \in R} \{ \pi_1(\mathbf{x}_A) \} \cup \{ \pi_1(\mathbf{x}_B) \}, \quad P_R = \bigcup_{(\mathbf{x}_A, \mathbf{x}_B) \in R} \left( \pi_1(\mathbf{x}_B), \pi_1(\mathbf{x}_A) \right),\end{equation*}
where $\pi_1 \colon \mathbb{R}^2 \rightarrow \mathbb{R}$ is the projection onto the first coordinate. In other words, we declare $x_B P_R x_A$ just when $(x_A, x_B)$ arises as the pair of $x_1$-coordinates of some pair $\left( \mathbf{x}_A, \mathbf{x}_B\right) \in R$.

\end{definition}

This association has the property that the winning strategies under any ruleset $R$ correspond exactly to the order-preserving maps of $X_R$ into the unit interval.

\begin{proposition}
A strategy $f \colon X_R \rightarrow [0, 1]$ is dominant for Player 2 if and only if $x_B P_R x_A \implies f(x_B) > f(x_A)$.
\end{proposition}
\begin{proof}
We identify each map $f \colon X_R \rightarrow [0, 1]$ with that strategy which declares that any $\mathbf{x} = (x_1, x_2)$ resides in $B$ with probability $f(x_1)$.

Such a strategy is winning, moreover, if and only if, for each $\left( \mathbf{x}_A, \mathbf{x}_B\right) \in R$, the quantity
\begin{align*}
P(\text{win}) &= P(\text{chooses } \mathbf{x}_B) \cdot P(\text{states ``$B$''}) + P(\text{chooses } \mathbf{x}_A) \cdot P(\text{states ``$A$''}) \\
&= \frac{1}{2} \cdot f(x_B) + \frac{1}{2} \left( 1 - f(x_A) \right) \\
&= \frac{1}{2} + \frac{f(x_B) - f(x_A)}{2}
\end{align*}
strictly exceeds $\frac{1}{2}$. This is true if and only if $f(x_B) > f(x_A)$ for each pair $\left( \mathbf{x}_A, \mathbf{x}_B\right)$ permitted under $R$.
\end{proof}

A result of Herden \cite{Herden:1989aa} establishes necessary and sufficient conditions under which at least one such map exists:

\begin{theorem}[Herden]
An order-preserving map $f \colon X_R \rightarrow [0, 1]$ exists if and only if $P_R$ is acyclic and there exists a countable descending chain $\left\{ D_i \right\}_{i \in \mathbb{N}}$ of subsets of $X_R$ within which, for each pair $x_B P_R x_A$, some $D_i$ exists  which satisfies $x_B \in D_i$ but $x_A \not \in D_i$.
\end{theorem}
\begin{proof}
This is \cite[Prop. 3.1]{Herden:1989aa}.
\end{proof}

Cover's original game's corresponding ordered set is $\mathbb{R}$ with its usual ordering, the order-preserving maps of which into the unit interval are exactly the strictly increasing functions (cf. Strategy \ref{increasing} above). The unsolvable game, on the other hand, yields $X_R = \mathbb{R}$ with $P_R$ the diagonal relation $P_R = \left\{ (x, x) \in \mathbb{R}^2 \mid x \in \mathbb{R} \right\}$. This relation is evidently not acyclic.

\bibliography{bibliography} {}

\begin{thebibliography}{KRTV13}

\bibitem[Big74]{Biggs:1974aa}
Norman Biggs.
\newblock {\em Algebraic Graph Theory}, volume~67 of {\em Cambridge Tracts in
  Mathematics}.
\newblock Cambridge University Press, 1974.

\bibitem[Cov87]{Cover:1987aa}
Thomas~M. Cover.
\newblock {\em Open problems in communication and computation}, chapter 5.1.
  Pick the Largest Number.
\newblock Springer-Verlag, 1987.

\bibitem[CS00]{Clark:2000aa}
Michael Clark and Nicholas Shackel.
\newblock The two-envelope paradox.
\newblock {\em Mind}, 109(435), 415-442 2000.

\bibitem[Fer89]{Ferguson:1989aa}
Thomas~S. Ferguson.
\newblock Who solved the secretary problem?
\newblock {\em Statistical Science}, 4(3):282--296, 1989.

\bibitem[Gne94]{Gnedin:1994aa}
Alexander~V. Gnedin.
\newblock A solution to the game of googol.
\newblock {\em The Annals of Probability}, 22(3):1588--1595, 1994.

\bibitem[Gne16]{Gnedin:2016aa}
Alexander Gnedin.
\newblock Guess the larger number.
\newblock {\em Mathematica Applicanda}, 44(1):183--207, 2016.

\bibitem[Her89]{Herden:1989aa}
G.~Herden.
\newblock On the existence of utility functions.
\newblock {\em Mathematical Social Sciences}, 17:297--313, 1989.

\bibitem[KRTV13]{Kesselheim:2013aa}
Thomas Kesselheim, Klaus Radke, Andreas T{\"o}nnis, and Berthold V{\"o}cking.
\newblock An optimal online algorithm for weighted bipartite matching and
  extensions to combinatorial auctions.
\newblock In Hans~L. Bodlaender and Giuseppe~F. Italiano, editors, {\em
  Algorithms -- ESA 2013}, volume 8125 of {\em Lecture Notes in Computer
  Science}. Springer, 2013.

\bibitem[LC98]{Lehmann:1998aa}
E.~L. Lehmann and George Casella.
\newblock {\em Theory of Point Estimation}.
\newblock Springer Texts in Statistics. Springer, second edition, 1998.

\bibitem[Sam81]{Samuels:1981aa}
S.~M. Samuels.
\newblock Minimax stopping rules when the underlying distribution is uniform.
\newblock {\em Journal of the American Statistical Association}, 76:188--197,
  1981.

\bibitem[Sed77]{Sedgewick:1977aa}
Robert Sedgewick.
\newblock Permutation generation methods.
\newblock {\em ACM Computing Surveys}, 9(2):137--164, 1977.

\bibitem[SN90]{Silverman:1990aa}
Stephen Silverman and Arthur N{\'a}das.
\newblock On the game of googol as the secretary problem.
\newblock In Thomas~F. Bruss, Thomas~S. Ferguson, and Stephen~M. Samuels,
  editors, {\em Strategies for Sequential Search and Selection in Real Time},
  volume 125 of {\em Contemporary Mathematics}. American Mathematical Society,
  1990.

\bibitem[SSS04]{Samet:2004aa}
Dov Samet, Iddo Samet, and David Schmeidler.
\newblock One observation behind two-envelope puzzles.
\newblock {\em The American Mathematical Monthly}, 111(4):347--351, 2004.

\end{thebibliography}
\bibliographystyle{alpha}

\textsc{Benjamin E. Diamond}
\newline JPMorgan Chase \& Co.\par\nopagebreak
email: \texttt{benediamond@gmail.com}

\end{document}